\documentclass[a4paper]{article}

\usepackage{amssymb,amsmath,amsfonts,amsthm,enumerate}
\usepackage{subfigure}
\usepackage{graphicx}
\usepackage{epsfig}
\usepackage{float}
\usepackage[usenames]{color}
\usepackage{cite}

\def\pa{\partial}
\def\nb{\nabla}
\def\Ga{\Gamma}
\def\Om{\Omega}
\def\al{\alpha}
\def\be{\beta}
\def\ga{\gamma}
\def\de{\delta}
\def\la{\lambda}
\def\vp{\varphi}
\def\om{\omega}
\def\BR{{\mathbb R}}
\def\f{\frac}
\def\rd{\mathrm d}

\newcommand{\lc}
{\mathrel{\raise2pt\hbox{${\mathop<\limits_{\raise1pt\hbox
{\mbox{$\sim$}}}}$}}}

\newcommand{\gc}
{\mathrel{\raise2pt\hbox{${\mathop>\limits_{\raise1pt\hbox{\mbox{$\sim$}}}}$}}}

\newcommand{\ec}
{\mathrel{\raise2pt\hbox{${\mathop=\limits_{\raise1pt\hbox{\mbox{$\sim$}}}}$}}}

\newtheorem{lemma}{Lemma}[section]
\newtheorem{theorem}{Theorem}[section]
\newtheorem{remark}{Remark}[section]
\newtheorem{prob}{Problem}[section]
\newtheorem{algorithm}{Algorithm}[section]
\newtheorem{example}{Example}[section]

\allowdisplaybreaks[4]

\numberwithin{equation}{section}

\title{Numerical Reconstruction of the Spatial Component in the Source Term of a Time-Fractional Diffusion Equation}

\author{Daijun Jiang\footnote{School of Mathematics and Statistics \& Hubei Key Laboratory of Mathematical Sciences, Central China Normal University, Wuhan 430079, China. E-mail: jiangdaijun@mail.ccnu.edu.cn}\and
Yikan Liu\footnote{Corresponding author. Research Center of Mathematics for Social Creativity, Research Institute for Electronic Science, Hokkaido University, N12W7, Kita-Ward, Sapporo 060-0812, Japan. E-mail: ykliu@es.hokudai.ac.jp}\and
Dongling Wang\footnote{Department of Mathematics and Center for Nonlinear Studies, Northwest University, Xi'an, Shaanxi 710075, China. E-mail: wdymath@nwu.edu.cn}}

\date{}

\begin{document}
\maketitle

\centerline{\bf ABSTRACT}
\begin{center}\begin{minipage}[t]{13cm} 
In this article, we are concerned with the analysis on the numerical reconstruction of the spatial component in the source term of a time-fractional diffusion equation. This ill-posed problem is solved through a stabilized nonlinear minimization system by an appropriately selected Tikhonov regularization. The existence and the stability of the optimization system are demonstrated. The nonlinear optimization problem is approximated by a fully discrete scheme, whose convergence is established under a novel result verified in this study that the $H^1$-norm of the solution to the discrete forward system is uniformly bounded. The iterative thresholding algorithm is proposed to solve the discrete minimization, and several numerical experiments are presented to show the efficiency and the accuracy of the algorithm.\medskip

{\bf Keywords}\ \ Time-fractional diffusion equation, Inverse source problem, Finite element method, Iterative thresholding algorithm\medskip

{\bf Mathematics subject classification (2010)}\ \ 35R11, 65M32, 41A35
\end{minipage}\end{center}

\section{Introduction}\label{sec:introduction}
\setcounter{equation}{0}

Let $T>0$ and $\Om\subset\BR^d$ ($d=1,2,3$) be a bounded convex polygonal domain. For $0<\al<1$, by $\pa_t^\al$ we denote the Caputo derivative defined as (see, e.g., \cite[p.\! 91]{KST06})
\[\pa_t^\al g(t):=\f1{\Ga(1-\al)}\int_0^t\f{g'(s)}{(t-s)^\al}\,\rd s,\]
where $\Ga(\,\cdot\,)$ is the usual Gamma function. In this paper, we consider the following initial-boundary value problem for a time-fractional diffusion equation (TFDE) with the homogeneous Neumann boundary condition
\begin{equation}\label{ch1}
\begin{cases}
(\pa_t^\al u-\triangle+1)u(x,t)=f(x)\mu(t), & (x,t)\in Q:=\Om\times(0,T),\\
u(x,0)=0, & x\in\Om,\\
\pa_\nu u(x,t)=0, & (x,t)\in\pa\Om\times(0,T),
\end{cases}
\end{equation}
where $\pa_\nu u:=\nb u\cdot\nu$ and $\nu=(\nu_1,\ldots,\nu_d)$ denotes the unit outward normal derivative on the boundary $\pa\Om$. The source term in \eqref{ch1} takes the form of separated variables, in which the spatial component $f(x)$ models the spatial distribution e.g.\! of the contaminant source, and the temporal component $\mu(t)$ describes the time evolution pattern.

The governing equation in \eqref{ch1} is a typical representative of a wide range of TFDEs, which were proposed as a powerful candidate for describing anomalous diffusion phenomena in heterogenous media. In the past decades, TFDEs have been investigated thoroughly from both theoretical and numerical aspects. It reveals from \cite{SY11a,LRY16,JLLY17} and other literature that TFDEs resemble their integer counterpart (i.e., $\al=1$) in such qualitative aspects like analyticity and maximum principle, but show certain difference in the senses of limited smoothing effect, short-/long-time asymptotic behavior and weak unique continuation. For the numerical simulation of TFDEs, we refer e.g.\! to \cite{lin07,Li18,NK18,jin16,Stynes}.

Simultaneously, various kinds of inverse problems for TFDEs have also gathered consistent popularity within the last decade. Owing to the limited smoothing property of the forward problem, usually the ill-posedness of inverse problems for TFDEs is less severe than that for the case of $\al=1$, which can be witnessed from the backward problems. Nevertheless, due to the lack of techniques for analysis, for some inverse problems, one should relate TFDEs with other types of equations, so that one can indirectly obtain uniqueness and sometimes stability. For the inverse source problems on determining the spatial component, it turns out that the majority of existing literature dealt with the final observation data $u(\,\cdot\,,T)$ except for \cite{ZX11}. We refer to Sakamoto and Yamamoto \cite{SY11b} for the generic well-posedness, and \cite{WYH13,WW14} etc.\! for the numerical reconstruction by various regularization methods. Recently, by a newly established unique continuation property for TFDEs, Jiang et al. \cite{JLLY17} proved the uniqueness for the same problem by the partial interior observation and developed an iterative thresholding algorithm. On reconstructing the temporal component in the source term, we refer e.g.\! to \cite{RW16,WLL16,LZ17}. For a complete bibliography on inverse problems for TFDEs, see the topical review articles \cite{JR15,LiuLiY19,LiLiuY19,LY19}.

In the same direction of \cite{JLLY17}, in this paper we mainly focus on the numerical aspect of the following inverse source problem.

\begin{prob}\label{prob-ISP}
Let $\om\subset\Om$ be a nonempty open subdomain and $u(f)$ be the solution to \eqref{ch1}. Provided that $\mu$ is known on $[0,T]$, determine the spatial component $f$ in $\Om$ by the partial interior observation of $u(f)$ in $\om\times(0,T)$.
\end{prob}

In fact, here we can also consider a more general formulation such as
\[
\pa_t^\al u(x,t)-\sum_{i,j=1}^d\f\pa{\pa x_i}\left(a_{ij}(x)\f{\pa u(x,t)}{\pa x_j}\right)+c(x)u(x,t)=f(x)\mu(t),
\]
where we assume certain regularity of $a_{ij}$ and $c$, and $(a_{ij})_{d\times d}$ is a strictly positive-definite matrix on $\overline\Om\,$. However, we restrict ourselves to the formulation \eqref{ch1} not only for its simplicity in the numerical simulation, but also due to the belief that the underlying ill-posedness are essentially the same.

On the other hand, we notice that the iteration method proposed in \cite{JLLY17} lacks detailed analysis especially on the convergence issue, although the numerical performance was discussed. From the practical viewpoint, it is obligatory to reformulate Problem \ref{prob-ISP} in a discrete setting and investigate whether the resulting system inherits the corresponding properties of the continuous one. Moreover, we should verify the convergence of the discretized solution to the continuous one in some sense. This motivates the central topic of the present article, which turns out to be a very important and necessary supplementation to the investigation of Problem \ref{prob-ISP}.

To fully discretize system \eqref{ch1}, we shall employ the standard Galarkin method with piecewise linear finite element in space and the L1 scheme in time. This method is one of the most popular and successful numerical scheme for discretizing the subdiffusion problem, and it has been analyzed from various aspects. The optimal error estimate in $L^2(\Omega)$ norm with respect to the regularity of the problem solution is established in \cite{jin16} for uniform time step size. In order to deal with the weak singularity of time-fractional diffusion equations, the L1-type scheme on graded time meshes is employed in the discrete scheme, and the optimal error estimate in $L^{\infty}(\Omega)$ norm is established for $d=1$ in \cite{Stynes} and $d=2,3$ in \cite{NK18}. For the convergence analysis, in this paper we further establish the uniform $H^1{(\Omega)}$ error estimate on $[0, T]$  by using the Poincar\'e-Friedrichs inequality and the discrete fractional Gr\"onwall's inequality for the L1 method, which turns out to be novel to the best of our knowledge.

The remainder of this paper is organized as follows. Recalling key ingredients in theory, in Section \ref{sec-Tik} we interpret Problem \ref{prob-ISP} as an optimization problem with Tikhonov regularization, and demonstrate its regularizing effects. In Section \ref{sec-FEM}, we discretize the optimization problem by a fully discrete finite element approximation and study its basic properties. Next, Section \ref{sec-cov} is devoted to the convergence analysis of the solution to the discrete optimization problem. Finally, we propose the iterative thresholding algorithm to solve the discrete problem in Section \ref{sec-numer}, and implement several numerical examples to show the efficiency and accuracy of the algorithm.

\section{Preliminary and Tikhonov regularization}\label{sec-Tik}

In this section, we shall make general preparations and formulate Problem \ref{prob-ISP} stated in Section \ref{sec:introduction} as a stabilized minimization system, and establish the unique existence of the solution as well as the stability of the minimization formulation.

Let $L^2(\Om)$ be a usual $L^2$-space equipped with the inner product $(\,\cdot\,,\,\cdot\,)$, and $H^\al(0,T)$, $H^\ga(\Om)$ ($\ga\in\BR$) etc.\! denote Sobolev spaces (e.g., Adams \cite{adams75}). Throughout this paper, $C>0$ stands for generic constants which may change line by line. In a Banach space $X$, we denote the weak convergence of a sequence $\{z_n\}$ to $z$ by
\[
z_n\rightharpoonup z\quad\mbox{in }X\mbox{ as }n\to\infty.
\]
We first revisit some basic facts concerning the initial-boundary value problem \eqref{ch1}. For the solution regularity, we define the fractional Sobolev spaces ${_0}H^\al(0,T)$ as
\[
{_0}H^\al(0,T):=\left\{\!\begin{alignedat}2
& H^\al(0,T), & \quad & 0<\al<\f12,\\
& \left\{g\in H^\f12(0,T);\int_0^T\f{|g(t)|^2}t\,\rd t<\infty\right\}, & \quad & \al=\f12,\\
& \{g\in H^\al(0,T);\,g(0)=0\}, & \quad & \f12<\al<1.
\end{alignedat}\right.
\]
We also define the Riemann-Liouville integral operator $J_{0+}^\al$ and the backward Riemann-Liouville integral operator $J^\al_{T-}$ respectively as
\[
J_{0+}^\al g(t):=\f1{\Ga(\al)}\int_0^t\f{g(\tau)}{(t-\tau)^{1-\al}}\,\rd\tau,\quad J^\al_{T-}g(t)=\f1{\Ga(\al)}\int_t^T\f{g(\tau)}{(\tau-t)^{1-\al}}\,\rd\tau
\]
for $g\in L^2(0,T)$. We know that the Caputo derivative can be rewritten as $\pa_t^\al g=J_{0+}^{1-\al}\f\rd{\rd t}g$, and the following relation holds (see \cite[Lemma 4.1]{JLLY17})
\begin{equation}\label{li1}
\int_0^T(J_{0+}^\al g_1)\,g_2\,\rd t=\int_0^T g_1\,J^\al_{T-}g_2\,\rd t,\quad\forall\,g_1,g_2\in L^2(0,T).
\end{equation}
For later use, we show a technical lemma on the basis of \eqref{li1}

\begin{lemma}\label{lem-new}
Let $g\in{_0}H^\al(0,T)\cap C[0,T]$ and $h\in C^1[0,T]$ with $J_{T-}^{1-\al}h(T)=0$. Then
\[
\int_0^T(\pa_t^\al g)\,h\,\rd t=-g(0)\,J_{T-}^{1-\al}h(0)-\int_0^T g\,(J_{T-}^{1-\al}h)'\,\rd t.
\]
\end{lemma}

\begin{proof}
Pick a sequence $\{g_n\}\subset C^\infty[0,T]$ such that $g_n\to g$ in ${_0}H^\al(0,T)\cap C[0,T]$. Then for each $n=1,2,\ldots$, we utilize \eqref{li1} and calculate
\begin{align*}
\int_0^T(\pa_t^\al g_n)\,h\,\rd t & =\int_0^T(J_{0+}^{1-\al}g_n')\,h\,\rd t=\int_0^T g_n'\,J_{T-}^{1-\al}h\,\rd t=\left.g_n\,J_{T-}^{1-\al}h\right|_{t=0}^{t=T}-\int_0^T g_n\,(J_{T-}^{1-\al}h)'\,\rd t\\
& =-g_n(0)\,J_{T-}^{1-\al}h(0)-\int_0^T g_n\,(J_{T-}^{1-\al}h)'\,\rd t,
\end{align*}
where the last equality follows from the assumption $J_{T-}^{1-\al}h(T)=0$. Then the proof is completed by passing $n\to\infty$ on both sides of the above identity.
\end{proof}

On the basis of \cite[Lemma 2.4]{JLLY17} and \cite[Theorem 2.2(i)]{SY11a}, we summarize the well-posedness of the forward problem \eqref{ch1} as follows.

\begin{lemma}\label{lem:y1}
Let $f\in L^2(\Om)$ and $\mu\in L^\infty(0,T)$. Then the initial-boundary value problem \eqref{ch1} admits a unique solution $u(f)\in{_0}H^\al(0,T;L^2(\Om))\cap C([0,T];L^2(\Om))\cap L^2(0,T;H^2(\Om))$. Moreover, there exists a constant $C>0$ depending on $\Om,T,\al$ and $\mu$ such that
\[
\|u(f)\|_{H^\al(0,T;L^2(\Om))}+\|u(f)\|_{C([0,T];L^2(\Om))}+\|u(f)\|_{L^2(0,T;H^2(\Om))}\le C\|f\|_{L^2(\Om)}.
\]
\end{lemma}

In \cite[Theorem 2.2(i)]{SY11a}, it was proved that the solution $u\in C([0,T];H^{-2\ga}(\Om))$ for any $\ga>\f d4-1$ provided that the source term belongs to $L^\infty(0,T;L^2(\Omega))$. Since we restrict the spatial dimensions to $d=1,2,3$ throughout this paper, one can choose negative $\ga$ and thus $u(f)$ makes sense in $C([0,T];L^2(\Om))$ under the assumptions of Lemma \ref{lem:y1}.

For the theoretical aspect of Problem \ref{prob-ISP}, we recall the uniqueness result stated in \cite[Theorem 2.6]{JLLY17}.

\begin{lemma}\label{lem-ISP}
Let $\om\subset\Om$ be an arbitrarily chosen open subdomain and $u(f)$ be the solution to \eqref{ch1}. Assume that $f\in L^2(\Om)$ and $\mu\in C^1[0,T]$ with $\mu(0)\ne0$. Then $u=0$ in $\om\times(0,T)$ implies $f=0$ in $\Om$.
\end{lemma}

Suppose that we are given the noisy observation data $u^\de\in L^2(\om\times(0,T))$ in practice. Usually, $u^\de$ satisfies
\[
\left\|u^\de-u(f^*)\right\|_{L^2(\om\times(0,T))}\le\de,
\]
where $f^*\in L^2(\Om)$ and $\de>0$ stand for the true source term and the noise level, respectively. To deal with the ill-posedness of Problem \ref{prob-ISP}, we still adopt a classical Tikhonov regularization methodology as that in \cite{JLLY17} to consider the following optimization problem with Tikhonov regularization
\begin{equation}\label{c2}
\min_{f\in L^2(\Om)} J(f),\quad
J(f):=\left\|u(f)-u^\de\right\|_{L^2(\om\times(0,T))}^2+\be\|f\|_{L^2(\Om)}^2,
\end{equation}
where $\be>0$ is the regularization parameter, and $u=u(f)$ satisfies the initial condition $u(x,0)=0$ and the variational system associated with the equation \eqref{ch1}: for any $\vp\in H^\al(0,T;L^2(\Om))\cap L^2(0,T;H^1(\Om))$,
\begin{equation}\label{var1}
\int_0^T\!\!\!\int_\Om(\pa_t^\al u(f)\,\vp+\nb u(f)\cdot\nb\vp+u(f)\,\vp)\,\rd x\rd t=\int_0^T\!\!\!\int_\Om f\mu\,\vp\,\rd x\rd t.
\end{equation}

Regarding the minimizer of \eqref{c2}, we first show the following unique existence result.

\begin{theorem}\label{thm:exist}
For any $u^\de\in L^2(\om\times(0,T))$, there exists a unique minimizer $f^*\in L^2(\Om)$ to the optimization problem \eqref{c2}.
\end{theorem}

\begin{proof}
Since $J(f)$ is nonnegative, we know that $\inf_{f\in L^2(\Om)}J(f)$ is finite. Thus there exists a minimizing sequence $\{f^n\}\subset L^2(\Om)$ such that
\[
\lim_{n\to\infty}J(f^n)=\inf_{f\in L^2(\Om)}J(f).
\]
Then by the definition of $J(f^n)$, it is obvious that $\{f^n\}$ is uniformly bounded in $L^2(\Om)$. Therefore, there exist $f^*\in L^2(\Om)$ and a subsequence of $\{f^n\}$, still denoted by $\{f^n\}$, such that
\[
f^n\rightharpoonup f^*\quad\mbox{in }L^2(\Om)\mbox{ as }n\to\infty.
\]
We shall prove that $f^*$ is indeed the unique minimizer of \eqref{c2}.

Since each $f^n$ corresponds with a solution $u(f^n)$ to \eqref{ch1} with $f=f^n$, it follows immediately from Lemma \ref{lem:y1} that
the sequence $\{u(f^n)\}$ is also uniformly bounded in ${_0}H^\al(0,T;L^2(\Om))\cap L^2(0,T;H^2(\Om))$. This, together with the compact embedding $H^2(\Om)\subset\subset L^2(\Om)$, indicates the existence of some $u^*\in{_0}H^\al(0,T;L^2(\Om)\cap L^2(0,T;H^2(\Om))$ and a subsequence of $\{u(f^n)\}$, again still denoted by $\{u(f^n)\}$, such that
\begin{equation}\label{jc7}
u(f^n)\rightharpoonup u^*\quad\mbox{in }{_0}H^\al(0,T;L^2(\Om))\cap L^2(0,T;H^2(\Om)).
\end{equation}
We claim $u^*=u(f^*)$. Actually, we utilize the fact that
\begin{equation}\label{eq-var}
\int_0^T\!\!\!\int_\Om(\pa_t^\al u(f^n)\,\vp+\nb u(f^n)\cdot\nb\vp+u(f^n)\,\vp)\,\rd x\rd t=\int_0^T\!\!\!\int_\Om f^n\mu\,\vp\,\rd x\rd t
\end{equation}
for all $\vp\in L^2(0,T;H^1(\Om))$. Since \eqref{jc7} implies
\[
\pa_t^\al u(f^n)\rightharpoonup\pa_t^\al u^*,\quad\nb u(f^n)\rightharpoonup\nb u^*\quad\mbox{in }L^2(Q),
\]
we pass $n\to\infty$ in \eqref{eq-var} to obtain
\begin{equation}\label{var12}
\int_0^T\!\!\!\int_\Om(\pa_t^\al u^*\,\vp+\nb u^*\cdot\nb\vp+u^*\,\vp)\,\rd x\rd t=\int_0^T\!\!\!\int_\Om f^*\mu\,\vp\,\rd x\rd t
\end{equation}
for all $\vp\in L^2(0,T;H^1(\Om))$. Next, we shall prove $u^*(\,\cdot\,,0)=0$, which with \eqref{var12} and the definition \eqref{var1} of weak solutions implies $u^*=u(f^*)$.

Indeed, it follows from Lemma \ref{lem:y1} that $u(f^n)\in{_0}H^\al(0,T;L^2(\Om))\cap C([0,T];L^2(\Om))$, which satisfies the assumption of Lemma \ref{lem-new}. Then by taking $\psi\in C^1[0,T]$ with $J^{1-\al}_{T^-}\psi(T)=0$ and $v\in L^2(\Om)$ arbitrarily, we employ Lemma \ref{lem-new} to deduce
\[
\int_\Om\int_0^T\pa_t^\al u(f^n)\,\psi\,v\,\rd t\rd x=-\int_\Om u(f^n)(\,\cdot\,,0)\,J^{1-\al}_{T^-}\psi(0)\,v\,\rd x-\int_\Om\int_0^T u(f^n)\,\f\rd{\rd t}J^{1-\al}_{T^-}\psi\,v\,\rd t\rd x.
\]
Noting that $u(f^n)(x,0)=0$ and letting $n\to\infty$ in the above equation, we obtain
\begin{equation}\label{li2}
\int_\Om\int_0^T\pa_t^\al u^*\,\psi\,v\,\rd t\rd x=-\int_\Om\int_0^T u^*\,\f\rd{\rd t}J^{1-\al}_{T^-}\psi\,v\,\rd t\rd x.
\end{equation}
In a parallel manner, we also have
\[
\int_\Om\int_0^T\pa_t^\al u^*\,\psi\,v\,\rd t\rd x=-\int_\Om u^*(\,\cdot\,,0)\,J^{1-\al}_{T^-}\psi(0)\,v\,\rd x-\int_\Om\int_0^T u^*\,\f\rd{\rd t}J^{1-\al}_{T^-}\psi\,v\,\rd t\rd x
\]
for any $\psi\in C^1[0,T]$ with $J^{1-\al}_{T^-}\psi(T)=0$ and $v\in L^2(\Om)$, which, together with \eqref{li2}, implies that $u^*(\,\cdot\,,0)=0$.

By $f^n\rightharpoonup f^*$ in $L^2(\Om)$ and \eqref{jc7}, we employ the lower semi-continuity of the $L^2$-norm to conclude
\begin{align*}
J(f^*) & =\left\|u(f^*)-u^\de\right\|_{L^2(\om\times(0,T))}^2+\be\|f^*\|_{L^2(\Om)}^2\\
& \le\liminf_{n\to\infty}\left\|u(f^n)-u^\de\right\|_{L^2(\om\times(0,T))}^2+\be\liminf_{n\to\infty}\|f^n\|_{L^2(\Om)}^2\\
& \le\liminf_{n\to\infty}J(f^n)=\inf_{f\in L^2(\Om)}J(f),
\end{align*}
indicating that $f^*$ is indeed a minimizer to the optimization problem \eqref{c2}. Furthermore, the uniqueness of $f^*$ is readily seen from the convexity of $J(f)$.
\end{proof}

Next, we justify the stability of \eqref{c2}, namely, the minimization system \eqref{c2} is indeed a stabilization for Problem \ref{prob-ISP} with respect to the perturbation in observation data.

\begin{theorem}\label{thm:stability}
Let $\{u^\de_\ell\}\subset L^2(\om\times(0,T))$ be a sequence such that
\begin{equation}\label{eq-cov}
u^\de_\ell\to u^\de\quad\mbox{in }L^2(0,T;L^2(\om))\mbox{ as }\ell\to\infty,
\end{equation}
and $\{f^\ell\}$ be a sequence of minimizers of problems
\[\min_{f\in L^2(\Om)}J_\ell(f),\quad J_\ell(f):=\left\|u(f)-u_\ell^\de\right\|_{L^2(\om\times(0,T))}^2+\be\|f\|_{L^2(\Om)}^2,\quad\ell=1,2,\ldots.\]
Then $\{f^\ell\}$ converges strongly in $L^2(\Om)$ to the minimizer of \eqref{c2}.
\end{theorem}

\begin{proof}
The unique existence of each $f^\ell$ is guaranteed by Theorem \ref{thm:exist}. By definition, we have
\[
J_\ell(f^\ell)\le J_\ell(f),\quad\forall\,f\in L^2(\Om),
\]
which implies the uniform boundedness of $f^\ell$ in $L^2(\Om)$. Hence, there exist $f^*\in L^2(\Om)$ and a subsequence of $\{f^\ell\}$, still denoted by $\{f^\ell\}$, such that
\[
f^\ell\rightharpoonup f^*\quad\mbox{in }L^2(\Om)\mbox{ as }\ell\to\infty.
\]

Now it suffices to show that $f^*$ is indeed the unique minimizer of \eqref{c2}. Actually, repeating the same argument as that in the proof of Theorem \ref{thm:exist}, we can derive
\[
u(f^\ell)\rightharpoonup u(f^*)\quad\mbox{in }{_0}H^\al(0,T;L^2(\Om))\cap L^2(0,T;H^2(\Om))\mbox{ as }\ell\to\infty,
\]
up to taking a further subsequence. Combining the above convergence with \eqref{eq-cov}, we obtain
\[
u(f^\ell)-u_\ell^\de\rightharpoonup u(f^*)-u^\de\quad\mbox{in }L^2(0,T;L^2(\om))\mbox{ as }\ell\to\infty.
\]
Therefore, we get
\begin{equation}\label{li5}
\|u(f^*)-u^\de\|_{L^2(\om\times(0,T))}^2\le\liminf_{\ell\to\infty}\left(\|u(f^\ell)-u_\ell^\de\|_{L^2(\om\times(0,T))}^2\right).
\end{equation}
For any $f\in L^2(\Om)$, again we take advantage of the lower semi-continuity of the $L^2$-norm to deduce
\begin{align}
J(f^*) & =\left\|u(f^*)-u^\de\right\|_{L^2(\om\times(0,T))}^2+\be\|f^*\|_{L^2(\Om)}^2\nonumber\\
& \le\liminf_{\ell\to\infty}\left\|u(f^\ell)-u_\ell^\de\right\|_{L^2(\om\times(0,T))}^2+\be\liminf_{\ell\to\infty}\left\|f^\ell\right\|_{L^2(\Om)}^2\nonumber\\
& \le\liminf_{\ell\to\infty}\left(\left\|u(f^\ell)-u_\ell^\de\right\|_{L^2(\om\times(0,T))}^2+\be\left\|f^\ell\right\|_{L^2(\Om)}^2\right)\nonumber\\
& \le\lim_{\ell\to\infty}\left(\left\|u(f)-u_\ell^\de\right\|_{L^2(\om\times(0,T))}^2+\be\|f\|_{L^2(\Om)}^2\right)\nonumber\\
& =\left\|u(f)-u^\de\right\|_{L^2(\om\times(0,T))}^2+\be\|f\|_{L^2(\Om)}^2=J(f),\quad\forall\,f\in L^2(\Om),\label{li4}
\end{align}
which verifies that $f^*$ is the minimizer of \eqref{c2}.

Next, we shall prove $\{f^\ell\}$ converges to $f^*$ strongly in $L^2(\Om)$ by contradiction. Assuming that it is not true, then we know that $\{\|f^\ell\|_{L^2(\Om)}\}$ does not converge to $\|f^*\|_{L^2(\Om)}$. As $f^\ell\rightharpoonup f^*$ in $L^2(\Om)$, by the weak lower semi-continuity of the norm, we have
\[
\|f^*\|_{L^2(\Om)}\le\liminf_{\ell\to\infty}\|f^\ell\|_{L^2(\Om)}.
\]
Hence, setting $A:=\limsup_{\ell\to\infty}\|f^\ell\|_{L^2(\Om)}$, we get
\begin{equation}\label{li7}
A=\limsup_{\ell\to\infty}\|f^\ell\|_{L^2(\Om)}>\liminf_{\ell\to\infty}\|f^\ell\|_{L^2(\Om)}\ge\|f^*\|_{L^2(\Om)},
\end{equation}
and there exists a subsequence $\{f^m\}$ of $\{f^\ell\}$  such that
\[
A=\lim_{m\to\infty}\|f^m\|_{L^2(\Om)}.
\]
Now taking $f=f^*$ in \eqref{li4}, we find that
\begin{align*}
\|u(f^*)-u^\de\|_{L^2(\om\times(0,T))}^2+\be\|f^*\|_{L^2(\Om)}^2 & =\liminf_{m\to\infty}\left(\|u(f^m)-u_m^\de\|_{L^2(\om\times(0,T))}^2+\be\|f^m\|_{L^2(\Om)}^2\right)\\
& =\liminf_{m\to\infty}\|u(f^m)-u_m^\de\|_{L^2(\om\times(0,T))}^2+\be A^2.
\end{align*}
This, together with \eqref{li7}, implies that
\begin{align*}
\|u(f^*)-u^\de\|_{L^2(\om\times(0,T))}^2 & =\liminf_{m\to\infty}\|u(f^m)-u_m^\de\|_{L^2(\om\times(0,T))}^2+\be\left(A^2-\|f^*\|_{L^2(\Om)}^2\right)\\
& >\liminf_{m\to\infty}\|u(f^m)-u_m^\de\|_{L^2(\om\times(0,T))}^2,
\end{align*}
which contradicts with \eqref{li5}. The proof of Theorem \ref{thm:stability} is completed.
\end{proof}

\section{Fully Discrete Finite Element Approximation}\label{sec-FEM}

In this section we propose a fully discrete finite element method for approximating the nonlinear optimization problem \eqref{c2}. We first introduce some appropriate time and space discretization. For the space discretization, we consider a shape regular triangulation $\mathcal T_h$ of $\Om$ with a mesh size $h$, consisting of tetrahedral elements (see \cite{cia78}). Then we introduce the standard nodal finite element spaces of piecewise linear functions
\begin{align*}
V_h & =\left\{v_h\in H^1(\Om);\,v_h|_A\in P_1(A),\ \forall\,A\in\mathcal T_h\right\},\\
S_h & =\left\{f_h\in L^2(\Om);\,f_h|_A\in P_1(A),\ \forall\,A\in\mathcal T_h\right\},
\end{align*}
where $P_1(A)$ is the space of linear polynomials on $A$. It is obvious that $V_h\subset S_h$.

To fully discretize the minimization problem \eqref{c2}, we also need the time discretization. To this end, we divide the time interval $[0,T]$ into $M$ subintervals by the equidistant nodal points
\[
0=t_0<t_1<\cdots<t_M=T
\]
with $t_m=m\,\tau$, where $\tau=T/M$ is the step length. Hereinafter, we additionally assume that the noisy observation data $u^\de$ satisfies $u^\de\in C([0,T];L^2(\om))$, so that each $u^\de(\,\cdot\,,t_m)$ makes sense in $L^2(\om)$.

To discretize the Caputo derivative in time, we refer to \cite{lin07,jin15} and use the following approximation:
\begin{align}
\pa_t^\al u(x,t_{m+1}) & =\f1{\Ga(1-\al)}\sum_{j=0}^m\int_{t_j}^{t_{j+1}}\f{\pa_su(x,s)}{(t_{m+1}-s)^\al}\,\rd s\nonumber\\
& =\f1{\Ga(1-\al)}\sum_{j=0}^m\f{u(x,t_{j+1})-u(x,t_j)}s\int_{t_j}^{t_{j+1}}(t_{m+1}-s)^{-\al}\,\rd s+r_\tau^{m+1}\nonumber\\
&=\f1{\Ga(2-\al)}\sum_{j=0}^md_j\f{u(x,t_{m+1-j})-u(x,t_{m-j})}{\tau^\al}+r_\tau^{m+1}\nonumber\\
& =\f1{\Ga(2-\al)\,\tau^\al} \sum_{j=0}^{m+1} \gamma_j u(x,t_{m+1-j})+r_\tau^{m+1}
:=\bar\pa_\tau^\al u(x,t_{m+1}) +r_\tau^{m+1},\label{jc8}
\end{align}
where $d_j:=(j+1)^{1-\al}-j^{1-\al}$ for $j=0,1,\ldots,m$, and $\bar\pa_\tau^\al u(x,t_{m+1})$ and $r_\tau^{m+1}$ denote the L1 numerical approximation\footnote{For historical reasons, the numerical scheme in \eqref{jc8} is often referred to as L1 numerical approximation for Caputo fractional derivative. One possible explanation is that in this scheme we replace the first derivative with a first order linear interpolation in each cell $[t_j,t_{j+1}]$, and is called L1 scheme. This name has nothing to do with the usual notation $L^1(\Om)$ for integrable function in $\Om$.} and local truncation error respectively.
The coefficients $\ga_j$ ($j=0,1,\ldots,m+1$) in the discrete convolution quadrature \eqref{jc8} are given by
\[
\gamma_{j}=\begin{cases}
1, & j=0,\\
(j+1)^{1-\al}-2j^{1-\al}+(j-1)^{1-\al}, & j=1,\ldots,m,\\
m^{1-\al}-(m+1)^{1-\al}, & j=m+1.
\end{cases}
\]

Following the same line as that in \cite{jin15}, now we arrive at a fully discrete scheme for the forward problem \eqref{ch1}:
given $u_h^0=0$, find $u_h^{m+1}\in V_h$ ($m=0,\ldots,M-1$) such that
\begin{equation}\label{zou3}
(1+b_0)\left(u_h^{m+1},\chi_h\right)+b_0\left(\nb u_h^{m+1},\nb\chi_h\right)=\sum_{j=0}^{m-1}(d_j-d_{j+1})\left(u_h^{m-j},\chi_h\right)+b_0\left(f_h\,\mu^{m+1},\chi_h\right)
\end{equation}
or equivalently, find $u_h^m\in V_h$ ($m=1,\ldots,M$) such that
\begin{equation}\label{bu1}
\left(\bar\pa_\tau^\al u_h^m, \chi_h\right)+(\nb u_h^m, \nb\chi_h)+ (u_h^m,\chi_h)=(f_h\,\mu^m, \chi_h)
\end{equation}
holds for all $\chi_h\in V_h$, where $b_0:=\Ga(2-\al)\,\tau^\al$. Here and hereinafter, we understand $f_h\in S_h$ as
some approximation of $f$ and $\mu^m=\mu(t_m)$.
To emphasize the dependency, we also denote the solution to \eqref{zou3} or \eqref{bu1} as $u_h^m(f_h)$.

On the basis of \eqref{zou3}, now we are well prepared to propose the fully discrete finite element approximation of the nonlinear optimization \eqref{c2}:
\begin{equation}\label{zou1}
\min_{f_h\in S_h}J_{h,\tau}(f_h),\quad J_{h,\tau}(f_h):=\tau\sum_{m=0}^Mc_m\left\|u_h^m(f_h)-u^\de(\,\cdot\,,t_m)\right\|_{L^2(\om)}^2+\be\|f_h\|^2_{L^2(\Om)}.
\end{equation}
Here $c_m$ ($m=0,1,\ldots,M$) are the coefficients of the composite trapezoidal rule for the time integration over $[0,T]$, i.e., $c_0=c_M=\f12$ and $c_m=1$ for $m=1,\ldots,M-1$.

Before verifying the existence of a minimizer to the discrete minimization problem \eqref{zou1}, we shall derive some useful a priori estimates for the discrete solution $u_h^m(f_h)$ to \eqref{zou3}
or \eqref{bu1}. To do so, we need the following lemmas which allows us to use the energy-like methods in the analysis.

\begin{lemma}[see \cite{WDL18}]\label{lem:Wang}
For the L1 numerical approximation \eqref{jc8} of the Caputo derivative, there holds
\begin{equation}\label{wang01}
\bar\pa_\tau^\al\|u_h^m\|_{L^2(\Om)}^2\le\left(2u_h^m,\bar\pa_\tau^\al u_h^m\right),\quad m\ge 1.
\end{equation}
\end{lemma}

\begin{lemma}[see \cite{Li18}]\label{lem:Li18}
Suppose that the nonnegative sequences $\{v^n\}$ and $\{g^n\}$ satisfy
\[
\bar\pa_\tau^\al v^n\le \la_1 v^n+  \la_2 v^{n-1}+g^n,\quad n\ge 1,
\]
where $\la_1,\la_2$ are given positive constants independent of the time step $\tau$.
Then there exists a positive constant $\tau_{*}>0$ such that
\[
v^n\le 2\left(v^{0}+\f{t_n^\al}{\Ga(1+\al)} \max_{0\le j\le n} g^{j}\right) E_\al(2\la\,t_n^\al)
\]
for $0<\tau<\tau_{*}$, where $E_\al(z)= \sum_{k=0}^{\infty} \f{z^{k}}{\Ga(1+k\al)}$ is the Mittag-Leffler function and $\la=\la_1+\f{\la_2}{2-2^{1-\al}}$.
\end{lemma}

Just as the exponential function plays a fundamental role in the integral order equation, the Mittag-Leffler function naturally appears in the fractional order calculus, and it has many nice properties similar to the exponential function. In fact, we have $E_1(\pm z)=\mathrm e^{\pm z}$ for all $z\in\mathbb C$. Let us mention a couple of properties that we will use later, and refer to the recent monograph \cite{GKMR14} for further details.

For $\al>0$, the Mittag-Leffler function $E_\al(z)$ is an entire function. In particular, for real variable $\xi\in\BR$ and $\al\in(0,1)$, $E_\al(\xi)$ is strictly positive and monotone increasing. Moreover, we have $E_\al(\xi)>0$, $E_\al(0)=1$,
\[
\lim_{\xi\to-\infty}E_\al(\xi)=0,\quad\lim_{\xi\to+\infty}E_\al(\xi)=+\infty,\quad\f\rd{\rd\xi}E_\al(\xi)=\f1\al E_{\al,\al}(\xi)>0,
\]
where $E_{\al,\al}(\xi):=\sum_{k=0}^\infty\f{\xi^k}{\Ga(\al+k\al)}$ is the two parameter Mittag-Leffler function. These properties can be derived from the definition of $E_\al(\xi)$ and the basic properties of gamma function $\Ga(\xi)$ for $\xi>0$. In view of the above properties for Mittag-Leffler functions, we know that the function $E_\al(2\la\,t_n^\al)$ involved in Lemma \ref{lem:Li18} is greater than one.

Now we proceed to the proof the uniform $H^1$-norm estimate for the solution $u_h^m(f_h)$ to the fully discrete scheme \eqref{zou3} or \eqref{bu1}.

\begin{lemma}\label{thm:forward}
The fully discrete scheme \eqref{zou3} or \eqref{bu1} is unconditional stable. Moreover, the following
estimates hold for $m=1,\ldots,M$:
\begin{equation}\label{zou6}
(i)~ \|u_h^m(f_h)\|_{L^2(\Om)}\le C_0\|f_h\|_{L^2(\Om)},\quad(ii)~ \|\nb u_h^m(f_h)\|^2_{L^2(\Om)}\le C_1\|f_h\|^2_{L^2(\Om)},
\end{equation}
where $C_0, C_1>0$ are constants independent of $h$ and $\tau$.
\end{lemma}

\begin{proof}
We adopt an inductive argument to prove $(i)$ in \eqref{zou6}. First, for $m=1$, it follows immediately from \eqref{zou3} and $u_h^0=0$ that
\[
(1+b_0)\left(u_h^1(f_h),\chi_h\right)+b_0\left(\nb u_h^1(f_h),\nb\chi_h\right)=b_0\left(f_h\,\mu^1,\chi_h\right), \quad\forall\,\chi\in V_h.
\]
Substituting $\chi_h=u_h^1(f_h)$ into the above equation, we obtain
\begin{align*}
b_0\left\|u_h^1(f_h)\right\|_{H^1(\Om)}^2 & \le(1+b_0)\left\|u_h^1(f_h)\right\|_{L^2(\Om)}^2+b_0\left\|\nb u_h^1(f_h)\right\|_{L^2(\Om)}^2\\
& =b_0\left(f_h\,\mu^1,u_h^1(f_h)\right)\le C\,b_0\|f_h\|_{L^2(\Om)}\left\|u_h^1(f_h)\right\|_{H^1(\Om)},
\end{align*}
which implies
\begin{equation}\label{jcc1}
\left\|u_h^1(f_h)\right\|_{H^1(\Om)}\le C_0\|f_h\|_{L^2(\Om)}.
\end{equation}

Next, assuming that $(i)$ in \eqref{zou6} holds for some $m\ge1$, we shall prove that it also holds for $m+1$.
Choosing $\chi_h=u_h^{m+1}(f_h)$ in \eqref{zou3} and employing the monotone decreasing property of the sequence $\{d_j\}$, we obtain
\begin{align*}
& \quad\,\,(1+b_0)\left\|u_h^{m+1}(f_h)\right\|_{L^2(\Om)}^2+b_0\left\|\nb u_h^{m+1}(f_h)\right\|_{L^2(\Om)}^2\\
& \le\left(\sum_{j=0}^{m-1}(d_j-d_{j+1})\left\|u_h^{m-j}(f_h)\right\|_{L^2(\Om)}+C\,b_0\|f_h\|_{L^2(\Om)}\right)\left\|u_h^{m+1}(f_h)\right\|_{L^2(\Om)}\\
& \le(C_0(1-d_m)+C\,b_0)\|f_h\|_{L^2(\Om)}\left\|u_h^{m+1}(f_h)\right\|_{L^2(\Om)}.
\end{align*}
By the same argument as that in the proof of \cite[Lemma 4.1]{jin15}, we can suitably choose $C_0>0$ such that
\[C_0(1-d_m)+C\,b_0\le C_0\]
holds uniformly for all $m=1,2,\ldots$. Then we immediately obtain $\|u_h^{m+1}(f_h)\|_{L^2(\Om)}\le C_0\|f_h\|_{L^2(\Om)}$.

We now prove the assertion $(ii)$ in \eqref{zou6}. Taking $\chi_h=2\bar\pa_\tau^\al u_h^m$ in \eqref{bu1} yields
\begin{align}\label{bu2}
\left(\bar\pa_\tau^\al u_h^m, 2\bar\pa_\tau^\al u_h^m\right)+\left(\nb u_h^m, \nb (2\bar\pa_\tau^\al u_h^m)\right)
+ \left(u_h^m, 2\bar\pa_\tau^\al u_h^m\right)=\left(f_h\,\mu^m, 2\bar\pa_\tau^\al u_h^m\right),
\end{align}
where
\[
\left(\nb u_h^m, \nb (2\bar\pa_\tau^\al u_h^m)\right)= \left(\nb u_h^m, 2 \bar\pa_\tau^\al
\nb u_h^m\right)\ge\bar\pa_\tau^\al \| \nb u_h^m\|_{L^2(\Om)}^2
\]
by \eqref{wang01}. Then \eqref{bu2} leads to
\begin{align*}
2\left\|\bar\pa_\tau^\al u_h^m\right\|_{L^2(\Om)}^2+ \bar\pa_\tau^\al \| \nb u_h^m\|_{L^2(\Om)}^2
\le  2\left\| \bar\pa_\tau^\al u_h^m\right\|_{L^2(\Om)}^2
+ \|u_h^m\|_{L^2(\Om)}^2 +\|f_h\,\mu^m\|_{L^2(\Om)}^2.
\end{align*}
Hence, we have that
\begin{align*}
\bar\pa_\tau^\al\|\nb u_h^m\|_{L^2(\Om)}^2 & \le\|u_h^m\|_{L^2(\Om)}^2+\|f_h\,\mu^m\|_{L^2(\Om)}^2\le C_2\left(\|\nb u_h^m\|_{L^2(\Om)}^2+\left|\int_\Om u_h^m\,\rd x\right|^2\right)+\|f_h\,\mu^m\|_{L^2(\Om)}^2\nonumber\\
& \le C_2\|\nb u_h^m\|_{L^2(\Om)}^2+C\left(\|u_h^m\|_{L^2(\Om)}^2+\|f_h\|_{L^2(\Om)}^2\right)\le C_2\|\nb u_h^m\|_{L^2(\Om)}^2+C_3\|f_h\|_{L^2(\Om)}^2,
\end{align*}
where we  made use of the Poincar\'e-Friedrichs inequality
\[
\|u_h^m\|_{L^2(\Om)}\le C\left(\|\nb u_h^m\|_{L^2(\Om)}+\left|\int_\Om u_h^m\,\rd x\right|\right).
\]
Now applying Lemma \ref{lem:Li18} and \eqref{jcc1}, we obtain that
\begin{align*}
\|\nb u_h^m\|_{L^2(\Om)}^2 & \le2\left(\|\nb u_h^1\|_{L^2(\Om)}^2+\f{t_m^\al}{\Ga(1+\al)}C_3\|f_h\|_{L^2(\Om)}^2\right)E_\al(2C_2t_m^\al)\\
& \le2\left(C_0\|f_h\|_{L^2(\Om)}^2+\f{t_m^\al}{\Ga(1+\al)}C_3\|f_h\|_{L^2(\Om)}^2\right) E_\al(2C_2 t_m^\al)\\
& \le2\left(C_0+\f{T^\al}{\Ga(1+\al)}C_3\right)E_\al(2C_2T^\al)\|f_h\|_{L^2(\Om)}^2\le C_1\|f_h\|_{L^2(\Om)}^2.
\end{align*}
This completes the proof.
\end{proof}

The next theorem provides the existence of a solution to the discrete system \eqref{zou1}.

\begin{theorem}\label{thm:hexist}
For each fixed $\tau>0$ and $h>0$, there exists at least a minimizer to the discrete system \eqref{zou1}.
\end{theorem}

\begin{proof}
It is readily seen from the non-negativity of $J_{h,\tau}(f_h)$ that $\inf J_{h,\tau}(f_h)$ is finite. Then there exists a minimizing sequence $\{f_h^n\}\subset S_h$ such that
\[\lim_{n\to\infty}J_{h,\tau}(f_h^n)=\inf_{f_h\in S_h}J_{h,\tau}(f_h).\]
Then $\{f_h^n\}$ is uniformly bounded in $S_h\subset L^2(\Om)$. Therefore, there exist $f_h^*\in S_h$ and a subsequence of $\{f_h^n\}$, still denoted by $\{f_h^n\}$, such that
\[f_h^n\rightharpoonup f_h^*\quad\mbox{in }L^2(\Om)\mbox{ as }n\to\infty.\]
Similarly to the proof of Theorem \ref{thm:exist}, we shall prove that $f_h^*$ is a minimizer of \eqref{zou1}.

Since $\tau>0$ is now a fixed constant, if follows from Lemma \ref{thm:forward} that for each $m=1,\ldots,M$, the sequence $\{u^m_h(f_h^n)\}$ is uniformly bounded in $V_h\subset L^2(\Om)$ with respect to $n$. Again, for each $m=1,\ldots,M$, this indicates the existence of $u^m_{h,*}\in V_h$ and a subsequence of $\{u^m_h(f_h^n)\}$, still denoted by $\{u_h^m(f_h^n)\}$, such that
\[
u_h^m(f_h^n)\rightharpoonup u^m_{h,*}\quad\mbox{in }L^2(\Om)\mbox{ as }n\to\infty,\ m=1,\ldots,M.
\]
In view of the norm equivalence in finite-dimensional spaces, the above two weak convergence results are actually strong, i.e.,
\begin{equation}\label{jc2}
\begin{aligned}
& u_h^m(f_h^n)\to u^m_{h,*}\mbox{ in }H^1(\Om),\ m=1,\ldots,M,\\
& f_h^n\to f_h^*\mbox{ in }H^1(\Om)
\end{aligned}
\quad\mbox{as }n\to\infty.
\end{equation}
Now it suffices to show $u^m_{h,*}=u_h^m(f_h^*)$. In fact, by definition \eqref{zou3} and up to taking a further subsequence, we see that $\{u_h^m(f_h^n)\}_{m=1}^M$ satisfies
\[
(1+b_0)\left(u_h^{m+1}(f_h^n),\chi_h\right)+b_0\left(\nb u_h^{m+1}(f_h^n),\nb\chi_h\right)=\sum_{j=0}^{m-1}(d_j-d_{j+1})\left(u_h^{m-j}(f_h^n),\chi_h\right)+b_0\left(f_h^n\,\mu_h^{m+1},\chi_h\right)
\]
for all $\chi_h\in V_h$. Passing $n\to\infty$ in the above equation and using \eqref{jc2}, we obtain
\[
(1+b_0)\left(u^{m+1}_{h,*},\chi_h\right)+b_0\left(\nb u^{m+1}_{h,*},\nb\chi_h\right)=\sum_{j=0}^{m-1}(d_j-d_{j+1})\left(u^{m-j}_{h,*},\chi_h\right)+b_0\left(f_h^*\,\mu_h^{m+1},\chi_h\right)
\]
for all $\chi_h\in V_h$. By definition \eqref{zou3} again, we conclude $u^m_{h,*}=u_h^m(f_h^*)$ by setting $u_{h,*}^0=0$ artificially.

Finally, collecting the above results, we again employ the lower semi-continuity to deduce
\begin{align*}
J_{h,\tau}(f_h^*) & =\tau\sum_{m=0}^Mc_m\left\|u_h^m(f_h^*)-u^\de(\,\cdot\,,t_m)\right\|_{L^2(\om)}^2+\be\|f_h^*\|_{L^2(\Om)}^2\\
& =\lim_{n\to\infty}\tau\sum_{m=0}^Mc_m\left\|u_h^m(f_h^n)-u^\de(\,\cdot\,,t_m)\right\|_{L^2(\om)}^2+\be\liminf_{n\to\infty}\|f_h^n\|_{L^2(\Om)}^2\\
& \le\liminf_{n\to\infty}J_{h,\tau}(f_h^n)=\inf_{f_h\in S_h}J_{h,\tau}(f_h).
\end{align*}
Consequently, $f_h^*$ is indeed a minimizer of $J_{h,\tau}(f_h)$ over $S_h$.
\end{proof}

\section{Convergence of the Fully Discrete Approximation}\label{sec-cov}

This section is devoted to the convergence analysis of the fully discrete finite element approximation \eqref{zou1}. In order to relate \eqref{zou1} with the continuous minimization problem \eqref{c2}, we start with the introduction of a few auxiliary tools and useful results.

First, we recall the standard $L^2$ projection $P_h$: $L^2(\Om)\to V_h$ and Ritz projection $R_h$: $H^1(\Om)\to V_h$, defined respectively by
\begin{alignat*}2
(P_h\vp,\chi_h) & =(\vp,\chi_h), & \quad & \forall\,\chi_h\in V_h,\\
(\nb R_h\vp,\chi_h) & =(\nb\vp,\chi_h), & \quad & \forall\,\chi_h\in V_h.
\end{alignat*}
The operators  $P_h$ and $R_h$ satisfy the following approximation properties (see \cite{cia78,tho06}):
\begin{alignat}2
& \|(I-P_h)\vp\|_{H^\ga(\Om)}\le C\,h^{1-\ga}\|\vp\|_{H^1(\Om)}, & \quad & \forall\,\vp\in H^1(\Om),\ \ga\in[0,1],\label{jc5}\\
& \|(I-R_h)\vp\|_{L^2(\Om)}+h\|\nb(\vp-R_h\vp)\|_{L^2(\Om)}\le C\,h^\ga\|\vp\|_{H^\ga(\Om)}, & \quad & \forall\,\vp\in H^\ga(\Om),\ \ga\in[1,2].\label{jc3}
\end{alignat}

The next lemma is a useful classical approximation result (see \cite{xiezou05,zei85}).

\begin{lemma}\label{lem:classic}
For a Banach space $X$ and a function $g\in C([0,T];X)$, define a step function approximation by
\[S_\Delta g(x,t)=\sum_{m=1}^M\mathbf1_{(t_{m-1},t_m]}(t)\,g(x,t_m),\]
where $\mathbf1_{(t_{m-1},t_m]}$ is the characteristic function of $(t_{m-1},t_m]$. Then we have the convergence
\[\lim_{\tau\to0}\int_0^T\|S_\Delta g(\,\cdot\,,t)-g(\,\cdot\,,t)\|^2_X\,\rd t=0.\]
\end{lemma}

In order to demonstrate the convergence of the finite element approximation \eqref{zou1} to the continuous minimization problem \eqref{c2}, we shall turn to the following key error estimate, which is a straightforward consequence of \cite[Theorem 3.6]{jin16}.

\begin{lemma}\label{lem:initial}
Let $u(f)$ be the solution to system \eqref{ch1}, where $f\in L^2(\Om)$ and $\mu\in C^1[0,T]$. Then for the solution $\{u_h^m(P_hf)\}$ to
the fully discrete scheme \eqref{zou3} with $f_h=P_hf$, we have the following error estimate: for any $m=1,\ldots,M$, there holds
\begin{equation}\label{dai7}
\|u(f)(\,\cdot\,,t_m)-u_h^m(P_hf)\|_{L^2(\Om)}\le C\|f\|_{L^2(\Om)}\|\mu\|_{C^1[0,T]}\left(h^2|\ln h|^2+t_m^{\al-1}\tau\right).
\end{equation}
\end{lemma}

\begin{remark}\label{re1}
According to \cite[Theorem 3.6]{jin16}, the source term $f\,\mu$ should satisfy
\begin{equation}\label{jin16}
\int_0^t(t-s)^{\al-1}\|f\,\mu'(s)\|_{L^2(\Om)}\,\rd s<\infty
\end{equation}
for $0<t\le T$. This is automatically guaranteed by the assumption $f\in L^2(\Om)$ and $\mu\in C^1[0,T]$. Therefore, the term including \eqref{jin16} with $t=t_m$ in \cite[Theorem 3.6]{jin16} is of order $O(\tau)$, which is absorbed in the last term in \eqref{dai7}.
On the other hand, since $t_m^{\al-1}\tau\le\tau^\al\to0$ as $\tau\to0$, \eqref{dai7} immediately implies
\[\lim_{h,\tau\to0}\|u(f)(\,\cdot\,,t_m)-u_h^m(P_hf)\|_{L^2(\Om)}=0,\quad m=1,\ldots,M.\]
\end{remark}

The following lemma provides a crucial strong convergence.

\begin{lemma}\label{lem:integral1}
Let $\{f_h\}_{h>0}$ be a sequence in $S_h$ which converges weakly to some $f\in L^2(\Om)$ as $h\to0$. Let $\{u_h^m(f_h)\}$ and $u(f)$ be the solutions of \eqref{zou3} and \eqref{ch1}, respectively. Then the following convergence holds:
\begin{equation}\label{hk30}
\lim_{h,\tau\to0}\tau\sum_{m=0}^Mc_m\left\|u_h^m(f_h)-u^\de(\,\cdot\,,t_m)\right\|_{L^2(\om)}^2=\left\|u(f)-u^\de\right\|_{L^2(\om\times(0,T))}^2.
\end{equation}
\end{lemma}

\begin{proof}
For simplicity, in this proof we abbreviate $u^m:=u(f)(\,\cdot\,,t_m)$.

Since we have $u(f)\in C([0,T];L^2(\Om))$ by Lemma \ref{lem:y1}, first we can directly apply Lemma \ref{lem:classic} to obtain
\[
\lim_{\tau\to0}\tau\sum_{m=0}^Mc_m\left\|u^m-u^\de(\,\cdot\,,t_m)\right\|_{L^2(\om)}^2=\left\|u(f)-u^\de\right\|_{L^2(\om\times(0,T))}^2.
\]
Using this convergence, \eqref{hk30} follows immediately if we can demonstrate
\begin{equation}\label{hk33}
\lim_{h,\tau\to0}\tau\sum_{m=0}^Mc_m\left\|u_h^m(f_h)-u^\de(\,\cdot\,,t_m)\right\|_{L^2(\om)}^2=\lim_{\tau\to0}\tau\sum_{m=0}^Mc_m\left\|u^m-u^\de(\,\cdot\,,t_m)\right\|_{L^2(\om)}^2.
\end{equation}

In order to show \eqref{hk33}, we utilize the a priori estimates of $u_h^m(f_h)$ and $u^m$ to deduce
\begin{align*}
& \quad\,\left|\tau\sum_{m=0}^Mc_m\left\|u_h^m(f_h)-u^\de(\,\cdot\,,t_m)\right\|_{L^2(\om)}^2-\tau\sum_{m=0}^Mc_m\left\|u^m-u^\de(\,\cdot\,,t_m)\right\|_{L^2(\om)}^2\right|\\
& \le\left(\tau\sum_{m=0}^M\|u_h^m(f_h)-u^m\|_{L^2(\om)}^2\right)^\f12\left(\tau\sum_{m=0}^M\|u_h^m(f_h)+u^m-2\,u^\de(\,\cdot\,,t_m)\|_{L^2(\om)}^2\right)^\f12\\
& \le C\left(\tau\sum_{m=0}^M\|u_h^m(f_h)-u^m\|_{L^2(\Om)}^2\right)^\f12.
\end{align*}
Clearly, the convergence in \eqref{hk33} follows if there holds
\begin{equation}\label{des1}
\lim_{h,\tau\to0}\tau\sum_{m=0}^M\|u_h^m(f_h)-u^m\|_{L^2(\Om)}^2=0.
\end{equation}
To this end, we split the error $u_h^m(f_h)-u^m$ into two parts:
\[u_h^m(f_h)-u^m=(u_h^m(f_h)-u_h^m(P_hf))+(u_h^m(P_hf)-u^m)=:\eta_h^m+\theta_h^m.\]
From the estimate \eqref{dai7} and Remark \ref{re1}, we can easily see that $\|\theta_h^m\|_{L^2(\Om)}\to0$ as $h,\tau\to0$.

Now we concentrate on the estimate of $\eta_h^m$. By the definitions of $u_h^m(f_h)$ and $u_h^m(P_hf)$, it is readily seen that $\eta_h^m$ satisfies
\begin{equation}\label{jc6}
(1+b_0)\left(\eta_h^{m+1},\chi_h\right)+b_0\left(\nb\eta_h^{m+1},\nb\chi_h\right)=\sum_{j=0}^{m-1}(d_j-d_{j+1})\left(\eta_h^{m-j},\chi_h\right)+b_0\left((f_h-P_hf)\mu_h^{m+1},\chi_h\right)
\end{equation}
for all $\chi_h\in V_h$. By the assumption of Lemma \ref{lem:integral1} and \eqref{jc5}, we have
\[
f_h\rightharpoonup f,\quad P_hf\to f\quad\mbox{in }L^2(\Om)\mbox{ as }h\to0,
\]
which indicates $f_h-P_hf\rightharpoonup0$ in $L^2(\Om)$ as $h\to0$. Meanwhile, by Lemma \ref{thm:forward} and $f_h\,,P_hf\in S_h$, we estimate
\[
\|\eta_h^m\|_{H^1(\Om)}\le C\|f_h-P_hf\|_{L^2(\Om)}\le C,\quad m=1,\ldots,M.
\]
Then for each $m=1,\ldots,M$, there exist some element $\eta_*^m\in H^1(\Om)$ and a subsequence of $\{\eta_h^m\}$, still denoted by $\{\eta_h^m\}$, such that
\begin{equation}\label{eq-cov-eta}
\eta_h^m\rightharpoonup \eta_*^m\mbox{ in }H^1(\Om),\quad\eta_h^m\to\eta_*^m\mbox{ in }L^2(\Om)\quad\mbox{as }h\to0,\ m=1,\ldots,M.
\end{equation}
Picking any $\chi\in H^1(\Om)$, we take $\chi_h=R_h\chi\in V_h$ as the test function in \eqref{jc6}. Then it follows from \eqref{jc3} that $\chi_h\to\chi$ in $H^1(\Om)$ as $h\to0$. Together with \eqref{eq-cov-eta}, we derive
\[
\left(\nb\eta_h^{m+1},\nb\chi_h\right)=\left(\nb\eta_h^{m+1},\nb\chi\right)+\left(\nb\eta_h^{m+1},\nb(\chi_h-\chi)\right)\to\left(\nb\eta_*^{m+1},\nb\chi\right)\quad\mbox{as }h\to0.
\]
Using \eqref{eq-cov-eta} again, we pass $h\to0$ in equation \eqref{jc6} to deduce
\[
(1+b_0)\left(\eta_*^{m+1},\chi\right)+b_0\left(\nb\eta_*^{m+1},\nb\chi\right)=\sum_{j=0}^{m-1}(d_j-d_{j+1})\left(\eta_*^{m-j},\chi\right),
\]
which turns out to be the semidiscrete scheme in time for the original problem \eqref{ch1} with $f=0$. Then it follows immediately from \cite{lin07} that $\|\eta_*^m\|_{L^2(\Om)}=0$ or equivalently $\eta_*^m=0$.

Consequently, we obtain $\|\eta_h^m\|_{L^2(\Om)}\to0$ as $h\to0$, and finally conclude
\[
\|u_h^m(f_h)-u^m\|_{L^2(\Om)}\le\|\theta_h^m\|_{L^2(\Om)}+\|\eta_h^m\|_{L^2(\Om)}\to0
\]
as $h,\tau\to0$. This completes the proof of \eqref{des1}.
\end{proof}

We conclude this section with the convergence of the fully discrete finite element approximation \eqref{zou1} to the continuous minimization problem \eqref{c2}.

\begin{theorem}\label{thm:convergence}
Let $\{f_h^*\}_{h>0}$ be the minimizers to the discrete minimization problem \eqref{zou1}. Then there exists a subsequence of $\{f_h^*\}_{h>0}$ which converges weakly in $L^2(\Om)$ to the minimizer of the continuous problem \eqref{c2} as $h,\tau\to0$.
\end{theorem}

\begin{proof}
It is routine to get the boundedness of $\{f_h^*\}$ in $L^2(\Om)$. Thus there exist $f^*\in L^2(\Om)$ and a
subsequence of $\{f_h^*\}$, still denoted by $\{f_h^*\}$, such that $f_h^*\rightharpoonup f^*$ in $L^2(\Om)$ as $h\to0$.

Now it suffices to show that $f^*$ is the minimizer of the continuous problem \eqref{c2}. For any $f\in L^2(\Om)$, we have
from \eqref{jc5} that
\begin{equation}\label{qi1}
P_hf\in S_h,\quad\lim_{h\to0}\|f-P_hf\|_{L^2(\Om)}=0.
\end{equation}
 Since $f_h^*$ is the minimizer of $J_{h,\tau}(f_h)$ over $S_h$, we obtain
\begin{equation}\label{zou20}
J_{h,\tau}(f_h^*)\le J_{h,\tau}(P_hf).
\end{equation}
Then using Lemma \ref{lem:integral1}, \eqref{qi1}, \eqref{zou20} and the lower semi-continuity of the $L^2$-norm, we deduce
\begin{align*}
J(f^*) & =\left\|u(f^*)-u^\de\right\|_{L^2(\om\times(0,T))}^2+\be\|f^*\|_{L^2(\Om)}^2\\
& \le\lim_{h,\tau\to0}\tau\sum_{m=0}^Mc_m\left\|u_h^m(f_h^*)-u^\de(\,\cdot\,,t_m)\right\|_{L^2(\om)}^2+\be\liminf_{h\to0}\|f_h^*\|_{L^2(\Om)}^2\\
& \le\liminf_{h,\tau\to0}J_{h,\tau}(f^*_h)\le\liminf_{h,\tau\to0}J_{h,\tau}(P_hf)\\
& \le\lim_{h,\tau\to0}\left\{\tau\sum_{m=0}^Mc_m\left\|u_h^m(P_hf)-u^\de(\,\cdot\,,t_m)\right\|_{L^2(\om)}^2+\be\|P_hf\|_{L^2(\Om)}^2\right\}\\
& =\left\|u(f)-u^\de\right\|_{L^2(\om\times(0,T))}^2+\be\|f\|_{L^2(\Om)}^2=J(f),
\end{align*}
which implies that $f^*$ is indeed a minimizer of the continuous problem \eqref{c2}.
\end{proof}

\section{Iterative Thresholding Algorithm and Numerical Examples}\label{sec-numer}

In this section, we develop an efficient iterative thresholding algorithm to solve the discrete minimization problem \eqref{zou3}--\eqref{zou1} and present several numerical experiments to show its efficiency and accuracy.

Nearly all effective iterative methods for solving nonlinear optimization problems need the information of the derivatives of the concerned objective functional. We shall first derive the derivative of the discrete nonlinear functional $J_{h,\tau}(f_h)$. Since $u_h^m(f_h)$ is linear with respect to $f_h$ in view of \eqref{zou3}, we have $u_h^m(f_h)'p_h=u_h^m(p_h)$ for any $p_h\in S_h$. Hence, it is straightforward to obtain
\begin{equation}\label{cc1}
J_{h,\tau}'(f_h)p_h=2\tau\sum_{m=0}^Mc_m\int_\om\left(u_h^m(f_h)-u^\de\right)u_h^m(p_h)\,\rd x+2\be\,(f_h,p_h).
\end{equation}
Needless to say, it is extremely expensive to use the above formula directly to evaluate the derivatives, since computing the derivative at one fixed point $f_h$ needs to solve equation \eqref{zou3} once for every direction $p_h\in S_h$. In order to reduce the computational costs for computing the derivatives, we recall the backward Riemann-Liouville derivative
\[
D_t^\al g(t):=-\f1{\Ga(1-\al)}\f\rd{\rd t}\int_t^T\f{g(s)}{(s-t)^\al}\,\rd s=-\f\rd{\rd t}J^{1-\al}_{T-}g(t)
\]
and introduce the following adjoint system
\begin{equation}\label{cc2}
\begin{cases}
(D_t^\al-\triangle+1)v=\mathbf1_\om\left(u(f)-u^\de\right) & \mbox{in }Q,\\
J^{1-\al}_{T-}v=0 & \mbox{in }\Om\times\{T\},\\
\pa_\nu v=0 & \mbox{on }\pa\Om\times(0,T),
\end{cases}
\end{equation}
where $\mathbf1_\om$ denotes the characterization function of $\om$. As before, we still denote the solution of \eqref{cc2} as $v(f)$ to emphasize its dependency upon $f$.

\begin{lemma}
Let $f,p\in L^2(\Om)$ and $u,v$ be the solutions of systems \eqref{ch1} and \eqref{cc2} respectively. Then there holds
\[
\int_0^T\!\!\!\int_\om\left(u(f)-u^\de\right)u(p)\,\rd x\rd t=\int_0^T\!\!\!\int_\Om p\,\mu\,v(f)\,\rd x\rd t.
\]
\end{lemma}

\begin{proof}
From the variational system  \eqref{var1}, we have for any $\vp\in H^\al(0,T;L^2(\Om))\cap L^2(0,T;H^1(\Om))$ that
\begin{equation}\label{var2}
\int_0^T\!\!\!\int_\Om(\pa_t^\al u(p)\,\vp+\nb u(p)\cdot\nb\vp+u(p)\,\vp)\,\rd x\rd t=\int_0^T\!\!\!\int_\Om p\mu\,\vp\,\rd x\rd t.
\end{equation}
On the other hand, it is routine to get the variational system associated with \eqref{cc2} that for any $\psi\in H^\al(0,T;L^2(\Om))\cap L^2(0,T;H^1(\Om))$,
\begin{equation}\label{var3}
\int_0^T\!\!\!\int_\Om(D_t^\al v(f)\,\psi+\nb v(f)\cdot\nb\psi+v(f)\,\psi)\,\rd x\rd t=\int_0^T\!\!\!\int_\om\left(u(f)-u^\de\right)\psi\,\rd x\rd t.
\end{equation}
Now letting $\vp=v(f)$ and $\psi=u(p)$ in \eqref{var2} and \eqref{var3} respectively, we obtain
\begin{align}
& \quad\,\int_0^T\!\!\!\int_\Om\pa_t^\al u(p)\,v(f)\,\rd x\rd t-\int_0^T\!\!\!\int_\Om p\mu\,v(f)\,\rd x\rd t\nonumber\\
& =\int_0^T\!\!\!\int_\Om D_t^\al v(f)\,u(p)\,\rd x\rd t-\int_0^T\!\!\!\int_\om\left(u(f)-u^\de\right)u(p)\,\rd x\rd t.\label{li10}
\end{align}
Further, by the integration by parts, the relation \eqref{li1} and noting that
\[
u(p)(x,0)=0,\quad(J^{1-\al}_{T-}v(f))(\,\cdot\,,T)=0,
\]
we have
\begin{align*}
\int_\Om\int_0^T\pa_t^\al u(p)\,v(f)\,\rd t\rd x & =\int_\Om\int_0^T J^{1-\al}_{0^+}\pa_t u(p)\,v(f)\,\rd t\rd x=\int_\Om\int_0^T\pa_t u(p)\,J^{1-\al}_{T^-}v(f)\,\rd t\rd x\\
& =-\int_\Om\int_0^T u(p)\,\f\rd{\rd t}J^{1-\al}_{T^-}v(f)\,\rd t\rd x=\int_\Om\int_0^T u(p)\,D_t^\al v(f)\,\rd t\rd x.
\end{align*}
Substituting the above equality into \eqref{li10}, we arrive at the desired equality.
\end{proof}

Writing the above equality into its discrete counterpart as that in previous sections, we have
\[
\tau\sum_{m=0}^Mc_m\int_\om\left(u_h^m(f_h)-u^\de\right)u_h^m(p_h)\,\rd x=\tau\sum_{m=0}^Mc_m\,(\mu_h^mv^m_h(f_h),p_h).
\]
Therefore, we can rewrite \eqref{cc1} as
\[
J_{h,\tau}'(f_h)p_h=2\tau\sum_{m=0}^Mc_m\,(\mu_h^mv^m_h(f_h),p_h)+2\be\,(f_h,p_h)=2\left(\tau\sum_{m=0}^Mc_m\,\mu_h^mv_h^m(f_h)+\be\,f_h,p_h\right).
\]
Since $p_h\in S_h$ was chosen arbitrarily, the above equality gives the following necessary condition for some element $f_h^*\in S_h$ to be a minimizer of the discrete optimization problem \eqref{zou1}:
\[
\tau\sum_{m=0}^Mc_m\,\mu_h^mv^m_h(f_h^*)+\be\,f_h^*=0.
\]
Following the same line as that in \cite{dau04,JLLY17}, this variational equation can be solved by the following iterative thresholding algorithm
\begin{equation}\label{cc6}
f_h^{k+1}=\f L{L+\be}f_h^k-\f\tau{L+\be}\sum_{m=0}^M c_m\,\mu_h^mv^m_h(f_h^k),\quad k=0,1,\ldots,
\end{equation}
where $L>0$ is a tuning parameter for the convergence. The choice of $L$ should be larger than the operator norm of the forward operator
\[
\begin{aligned}
\mathcal A:L^2(\Om) & \longrightarrow L^2(\om\times(0,T)),\\
f & \longmapsto u(f)_{\om\times(0,T)},
\end{aligned}
\]
and we refer to \cite{JLLY17} for details. It is obvious that at each step of \eqref{cc6}, we only need to solve equation \eqref{zou3} for $u_h^m(f_h^k)$ and then the discrete formulation of system \eqref{cc2} for $v_h^m(f_h^k)$ subsequently, which does not involve heavy computational costs. By the way, although there appears the backward Riemann-Liouville derivative $D_{T-}^\al$ in \eqref{cc2}, we know that the solution $v(f)$ coincides with
the following problem with a backward Caputo derivative
\begin{equation}\label{cc7}
\begin{cases}
(\pa_{T-}^\al-\triangle+1)v=\mathbf1_\om\left(u(f)-u^\de\right) & \mbox{in }Q,\\
v=0 & \mbox{in }\Om\times\{T\},\\
\pa_\nu v=0 & \mbox{on }\pa\Om\times(0,T),
\end{cases}
\end{equation}
where
\[
\pa_{T-}^\al g(t):=-\f1{\Ga(1-\al)}\int_t^T\f{g'(s)}{(s-t)^\al}\,\rd s.
\]
Therefore, instead of \eqref{cc2} we just discretize system \eqref{cc7} to compute $v^m_h(f_h)$ by almost the same discrete method for $u^m_h(f_h)$.

We are now ready to propose the iterative thresholding algorithm for the reconstruction.

\begin{algorithm}\label{algori-itera}
Choose a tolerance $\epsilon>0$, a regularization parameter $\be>0$ and a tuning parameter $L>0$. Give an initial guess $f_h^0$, and set $k=0$.
\begin{enumerate}
\item Compute $f_h^{k+1}$ according to the iterative update \eqref{cc6}.
\item If $\|f_h^{k+1}-f_h^k\|_{L^2(\Om)}/\|f_h^k\|_{L^2(\Om)}<\epsilon$, stop the iteration. Otherwise, update $k\leftarrow k+1$ and return to Step 1.
\end{enumerate}
\end{algorithm}

Next we shall present several numerical experiments by applying Algorithm \ref{algori-itera} to solve Problem \ref{prob-ISP} for $d=1,2$. In general, we set $\Om=(0,1)^d$, $T=1$ and specify various parameters involved in Algorithm \ref{algori-itera} as follows. Without loss of generality, we always choose the initial guess $f_h^0$ as a constant, e.g., $f_h^0\equiv2$. With the true source term $f^*$ and thus the noiseless data $u(f^*)$, the noisy data $u^\de$ is generated as
\[
u^\de(x,t)=(1+\de\,\mathrm{rand}(-1,1))\,u(f^*)(x,t),\quad x\in\om,\ 0<t\le T,
\]
where $\mathrm{rand}(-1,1)$ denotes the random number uniformly distributed on $[-1,1]$. For simplicity, in all examples we set the regularization parameter $\be=10^{-4}$. Besides the illustrative figures, we mainly evaluate the numerical performance of Algorithm \ref{algori-itera} by the number $K$ of iterations and the relative $L^2$ error
\[
\mathrm{err}:=\f{\left\|f_h^K-f^*\right\|_{L^2(\Om)}}{\|f^*\|_{L^2(\Om)}},
\]
where we understand $f_h^K$ as the result of the numerical reconstruction.

In the one-dimensional case, we divide the space-time region $\overline\Om\times[0,T]=[0,1]^2$ into $40\times40$ equidistant meshes. We set the tuning parameter $L$, the stopping criteria $\epsilon$ and the known component $\mu(t)$ as $L=1$, $\epsilon=2\times10^{-3}$ and $\mu(t)=5+10\,t$ respectively in Algorithm \ref{algori-itera}. Except for the factors mentioned above, we shall test the numerical performance with different choices of exact source terms $f^*$, fractional orders $\al$, noise levels $\de$ and observation subdomains $\om$.

\begin{example}\label{ex1.1}
In this example, we fix the observation subdomain $\om$ and the noise level $\de$ as $\om=\Om\setminus[1/20,19/20]$ and $\de=1\%$, respectively.
We test Algorithm \ref{algori-itera} with different fractional orders $\al$ and exact source terms $f^*$ as follows:
\begin{enumerate}
\item[{\rm(a)}] $\al=0.3$, $f^*(x)=\sin(\pi x/2)+x^2+1$.
\item[{\rm(b)}] $\al=0.5$, $f^*(x)=\sin(\pi x)-2$.
\end{enumerate}
Figure \ref{a1} compares the true source terms $f^*$ with the corresponding reconstructions $f_h^K$, and also shows the iteration steps $K$ and the relative errors in the caption.
\begin{figure}[htbp]\centering
\includegraphics[trim=12mm 8mm 12mm 6mm,clip=true,width=.47\textwidth]{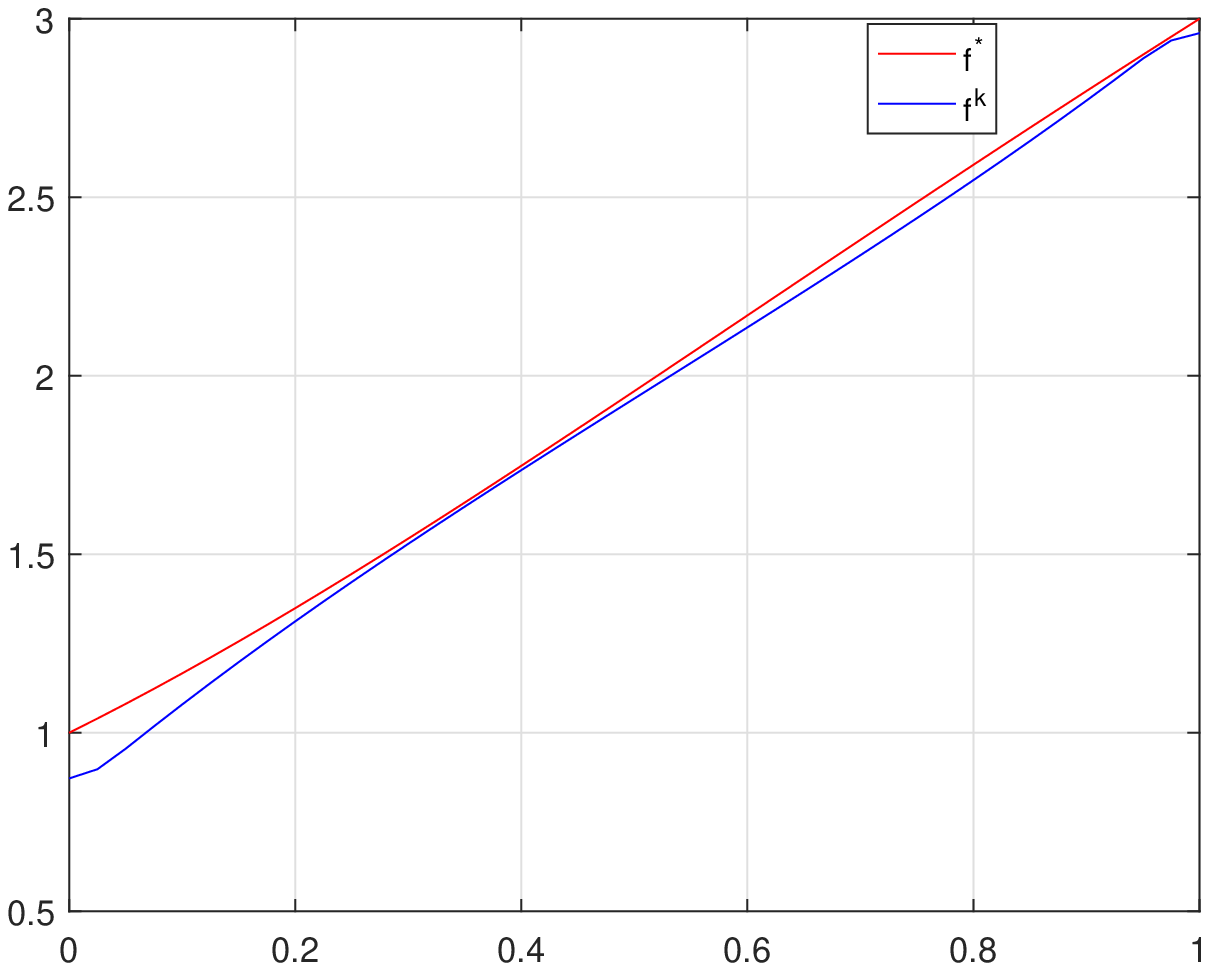}
\qquad
\includegraphics[trim=12mm 8mm 12mm 6mm,clip=true,width=.47\textwidth]{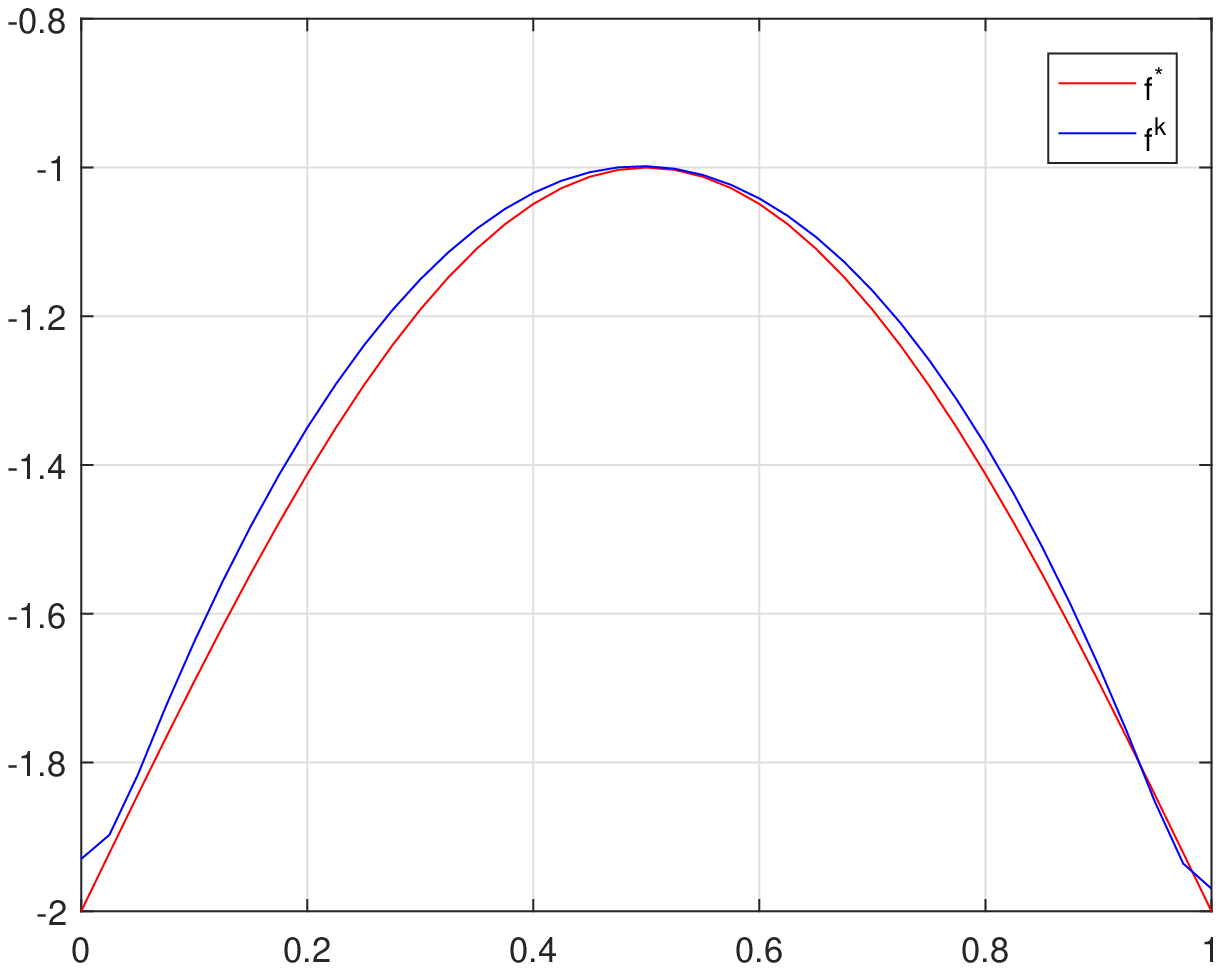}\\
\caption{True source terms $f^*$ and their reconstructions $f_h^K$ obtained in Example \ref{ex1.1}. Left: Case (a), $K=29$, $\mathrm{err}=2.53\%$. Right: Case (b), $K=4$, $\mathrm{err}=2.51\%$.}\label{a1}
\end{figure}
\end{example}

\begin{example}\label{ex1.5}
In this example, we fix $\al=0.8$ and the true source term $f^*(x)=-\sin(\pi x)+x+4$, and test different combinations of the observation subdomain $\om$ and the noise level $\de$ to observe their influence on the numerical performance. We first fix the noise level $\de=1\%$ and shrink the size of $\om$ from $\Om\setminus[1/10,9/10]$, $\Om\setminus[1/20,19/20]$ to $\Om\setminus[1/40,39/40]$. Next, we keep $\om=\Om\setminus[1/20,19/20]$ and enlarge the noise level $\de$ from $0.5\%$, $1\%$, $2\%$ to $4\%$. The resulting numerical performance of the reconstructions is listed in Table \ref{tab1}.
\begin{table}[htbp]\centering
\caption{Numerical performance in Example \ref{ex1.5} under various combinations of the observation subdomain $\om$ and the noise level $\de$.}\label{tab1}\medskip
\begin{tabular}{ll|ll}
\hline\hline
$\de$ & $\om$ & $\mathrm{err}$ & $K$\\
\hline
$1\%$ & $\Om\setminus[1/10,9/10]$ & $3.75\%$ & $13$\\
$1\%$ & $\Om\setminus[1/20,19/20]$ & $3.38\%$ & $18$\\
$1\%$ & $\Om\setminus[1/40,39/40]$ & $4.17\%$ & $27$\\
\hline
$0.5\%$ & $\Om\setminus[1/20,19/20]$ & $2.96\%$ & $16$\\
$2\%$ & $\Om\setminus[1/20,19/20]$ & $5.53\%$ & $20$\\
$4\%$ & $\Om\setminus[1/20,19/20]$ & $11.17\%$ & $24$\\
\hline\hline
\end{tabular}
\end{table}
\end{example}

We can see from Figures \ref{a1} that with different fractional orders $\al$ and $1\%$ noise in the observation data, the numerical reconstructions $f_h^K$ appear to be quite satisfactory in view of the ill-posedness of Problem \ref{prob-ISP}, regardless of the casual choice of the initial guess $f_h^0$. In addition, we can observe from Table \ref{tab1} that Algorithm \ref{algori-itera} has two important advantages, namely, its strong robustness against the oscillating noise in observation data, and its insensitivity to the smallness of the observation subdomain $\om$.

Now we consider the more challenging two-dimensional case, where we divide the space-time region $\overline\Om\times[0,T]=[0,1]^3$ into $40\times40\times20$ equidistant meshes. Here we set the tuning parameter $L$ and the known component $\mu(t)$
in Algorithm \ref{algori-itera} as $L=2$ and $\mu(t)=1+10\pi\,t^2$ respectively. Analogously to the one-dimensional counterpart, we evaluate the numerical performance of Algorithm \ref{algori-itera} from various aspects, including different combinations of the true source terms, noise levels and observation subdomains.

\begin{example}\label{ex2.1}
Parallel to Example \ref{ex1.1}, first we fix the observation subdomain $\om$ and the noise level $\de$ as $\om=\Om\setminus[1/10,9/10]^2$ and $\de=1\%$, respectively. We take the stopping criteria $\epsilon=\de/3$ and specify two pairs of fractional orders and true source terms as follows:
\begin{enumerate}
\item[{\rm(a)}] $\al=0.3$, $f^*(x)=\sin(x_1)+\sin(x_2)+1$.
\item[{\rm(b)}] $\al=0.5$, $f^*(x)=\cos(\pi x_1)\cos(\pi x_2)+2$.
\end{enumerate}
In Figure \ref{b1}, we illustrate the true source terms $f^*$ with the corresponding reconstructions $f_h^K$, and the iteration steps $K$ and the relative errors are shown in the caption.
\begin{figure}[htbp]\centering
\includegraphics[trim=12mm 6mm 12mm 8mm,clip=true,width=.47\textwidth]{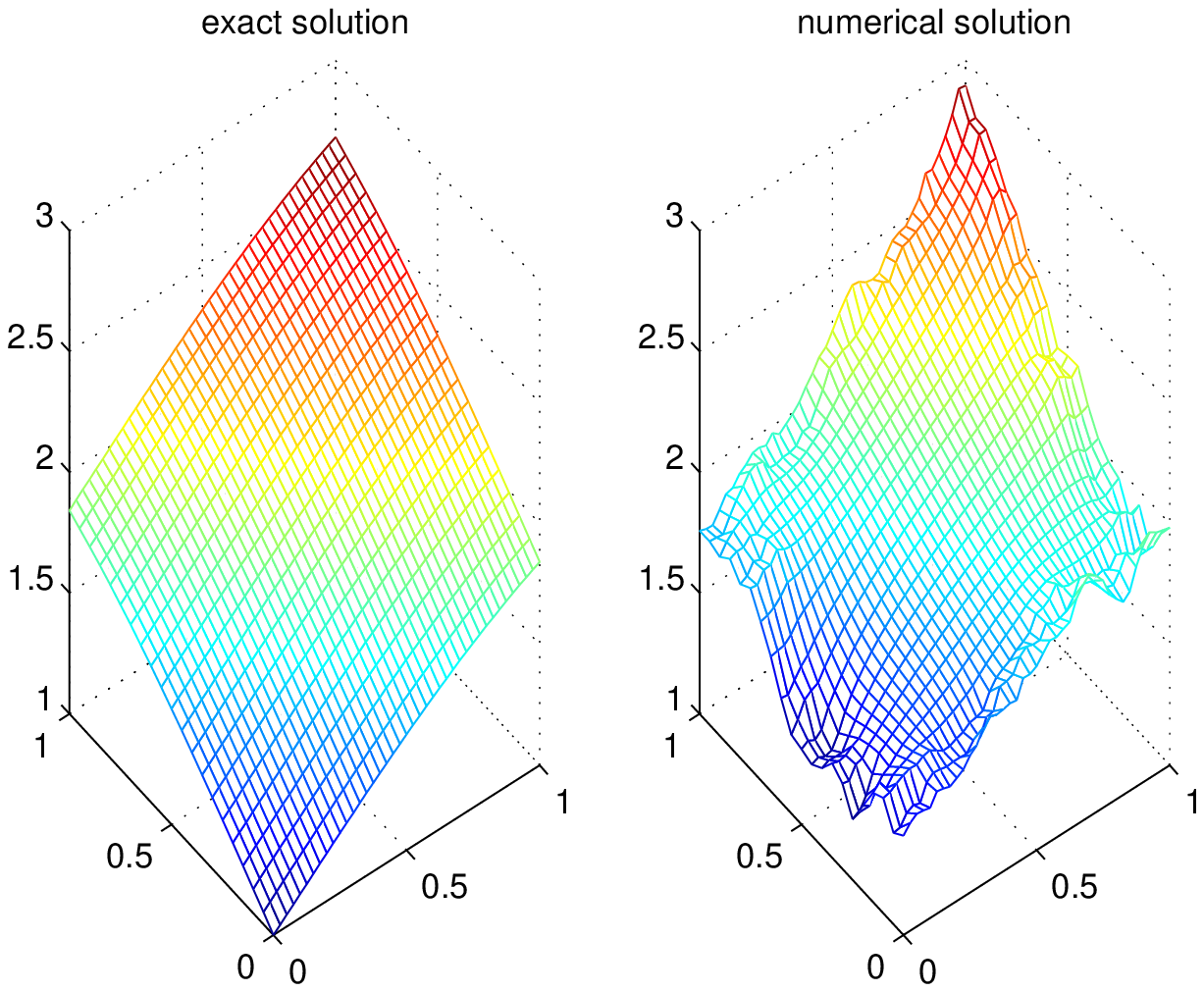}
\qquad
\includegraphics[trim=12mm 6mm 12mm 8mm,clip=true,width=.47\textwidth]{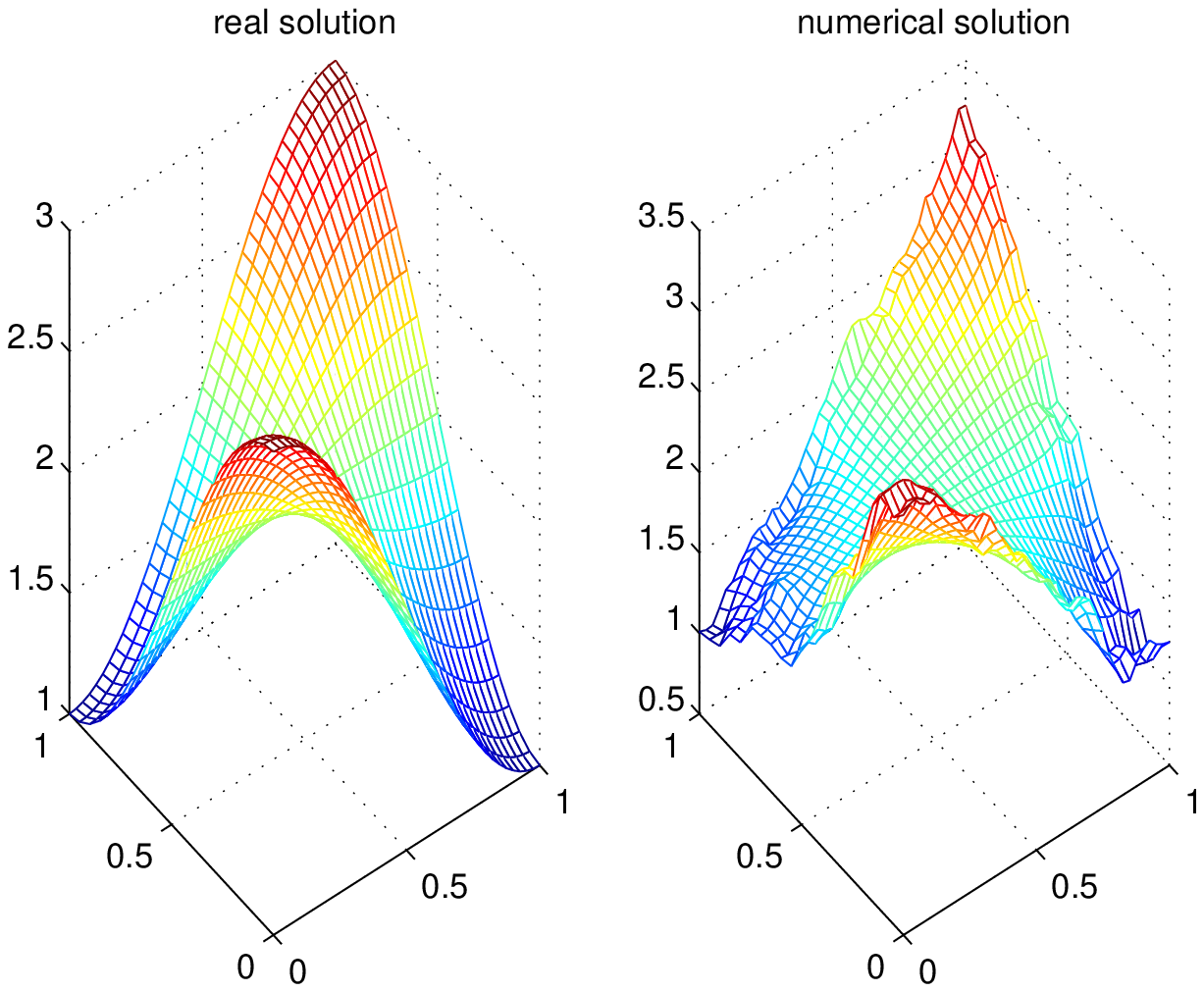}\\
\caption{True source terms $f^*$ and their reconstructions $f_h^K$ obtained in Example \ref{ex2.1}. Left: Case (a), $K=20$, $\mathrm{err}=6.26\%$. Right: Case (b), $K=36$,  $\mathrm{err}=7.17\%$.}\label{b1}
\end{figure}
\end{example}

\begin{example}\label{ex2.3}
In this example we fix $\al=0.8$, $f^*(x)=\exp((x_1+x_2)/4)+1$ and test Algorithm \ref{algori-itera} with various choices of noise levels $\de$ and observation subdomains $\om$ in a similar manner as that in Example \ref{ex1.5}. Here we set $\epsilon=\de/5$ as the stopping criteria. For the choice of $\om$, we not only adjust its size, but also change its coverage of the boundary $\pa\Om$. The resulting numerical performance of the reconstructions is listed in Table \ref{diff2}.
\begin{table}[htbp]\centering
\caption{Numerical performance in Example \ref{ex2.3} under various combinations of the observation subdomain $\om$ and the noise level $\de$.}\label{diff2}\medskip
\begin{tabular}{ll|ll}
\hline\hline
$\de$ & $\om$ & $\mathrm{err}$ & $K$\\
\hline
$1\%$ & $\Om\setminus[1/10,9/10]^2$ & $2.63\%$ & $13$\\
$1\%$ & $\Om\setminus[1/20,19/20]^2$ & $4.94\%$ & $27$\\
$1\%$ & $\Om\setminus[0,9/10]\times[1/10,9/10]$ & $4.36\%$ & $9$\\
\hline
$0.5\%$ & $\Om\setminus[1/10,9/10]^2$ & $2.47\%$ & $24$\\
$1\%$ & $\Om\setminus[1/10,9/10]^2$ & $3.61\%$ & $15$\\
$2\%$ & $\Om\setminus[1/10,9/10]^2$ & $5.06\%$ & $7$\\
$4\%$ & $\Om\setminus[1/10,9/10]^2$ & $6.58\%$ & $5$\\
\hline\hline
\end{tabular}
\end{table}
\end{example}

It is readily seen from Examples \ref{ex2.1}--\ref{ex2.3} that Algorithm \ref{algori-itera} also works efficiently and accurately in the two-dimensional case. It inherits almost all advantages witnessed in the one-dimensional tests in view of its strong robustness against oscillating noise as well as its insensitivity to the smallness of the observation subdomain.\bigskip

\noindent{\bf Acknowledgement}\ \ The authors thank Professor Bangti Jin (University College London) for his constructive discussions, and appreciate the  valuable comments by the anonymous referees.\bigskip

\noindent{\bf Funding information}\ \ The first author is supported by National Natural Science Foundation of China (NSFC, Nos. 11871240, 11401241, 11571265) and NSFC-RGC (China-Hong Kong, No. 11661161017).
The second author is supported by Japan Society for the Promotion of Science KAKENHI Grant Number JP15H05740.
The last author is supported by NSFC (Nos. 11871057 and 11501447).

\end{document}